\documentclass[12pt]{amsart}
\usepackage[mathscr]{eucal}
\usepackage{amscd}
\usepackage{amsfonts}
\usepackage{amsmath}
\usepackage{amsthm}
\usepackage{amssymb}
\usepackage{latexsym}
\usepackage[left=2cm,top=2cm,right=3cm,foot=2cm,headsep=1cm]{geometry}

\begin{document}

\newcommand{\nc}{\newcommand}
\newcommand{\bom}{{_{\mathbf{\omega}}}}
\newcommand{\st}{\divideontimes}
\def\neweq{\setcounter{theorem}{0}}
\newtheorem{theorem}[]{Theorem}
\newtheorem{proposition}[]{Proposition}
\newtheorem{corollary}[]{Corollary}
\newtheorem{lemma}[]{Lemma}
\newtheorem{example}[]{Example}
\theoremstyle{definition}
\newtheorem{definition}[]{Definition}
\newtheorem{remark}[]{Remark}
\newtheorem{conjecture}[equation]{Conjecture}
\newcommand{\dis}{{\displaystyle}}
\def\question{\noindent\textbf{Question.} }
\def\remark{\noindent\textbf{Remark.} }
\def\proof{\medskip\noindent {\textsl{Proof.} \ }}
\def\endproof{\hfill$\square$\medskip}
\def\str{\rule[-.2cm]{0cm}{0.7cm}}
\newcommand{\beq}{\begin{equation}\label}
\newcommand{\aand}{\quad{\text{\textsl{and}}\quad}}
\newcommand{\la}{\label}
\numberwithin{equation}{section}
\title{Kac's conjecture and the algebra of BPS states }
\author{Tim Cramer}
\thanks{\emph{E-mail address}: tim.cramer@yale.edu}
\maketitle
Let $\vec{Q}$ be an affine quiver and let $\mathfrak{n}$ be the positive part of the affine Lie algebra associated to $\vec{Q}$. We provide a construction of $\mathfrak{n}$ using the semistable irreducible components in the Lusztig nilpotent variety associated to $\vec{Q}$. This confirms the main conjecture of \cite{FMV2} on defining the so-called algebra of BPS states on the minimal resolution of a Kleinian singularity. Using the results of \cite{CBVdB}, we show that our construction is closely connected to Kac's constant term conjecture in the case of an affine quiver.

\section{Introduction}

Let $\Gamma$ be a finite subgroup of $SL(2,\mathbb{C})$. It is known that the Kleinian singularity $\mathbb{C}^2/\Gamma$ admits a unique crepant resolution $p:\widetilde{\mathbb{C}^2/\Gamma}\rightarrow  \mathbb{C}^2/\Gamma$. Moreover, the exceptional divisor $p^{-1}(0)$ is known to be a tree of $\mathbb{P}^1$'s with intersection matrix that is the negative of the Cartan matrix associated to some simply-laced Dynkin graph $Q$. The classical McKay correspondence relates $Q$ (or more precisely, the extended Dynkin diagram $\widehat{Q}$ obtained from $Q$) to the representation ring of the group $\Gamma$. Kronheimer \cite{Kr} and Cassens-Slodowy \cite{CS} have shown that the varieties $\mathbb{C}^2/\Gamma, \widetilde{\mathbb{C}^2/\Gamma},$ and $p^{-1}(0)$ can all be expressed as moduli spaces of representations of the preprojective algebra $\Pi_{\widehat{Q}}$ associated to $\widehat{Q}$. In particular, the exceptional divisor $p^{-1}(0)$ corresponds to the moduli space of nilpotent, semistable $\delta$-dimensional representations of $\Pi_{\widehat{Q}}$. Here $\delta$ is the first imaginary root associated to $\widehat{Q}$. For a general dimension vector $\alpha$, the moduli space of nilpotent $\alpha$-dimensional representations of $\Pi_{\widehat{Q}}$ is known as Lusztig's nilpotent variety $\Lambda_{\alpha}$. When $\alpha$ is indivisible, the number of semistable irreducible components in $\Lambda_{\alpha}$ is equal to the dimension of the root space $\mathfrak{g}_{\alpha}$ of the Kac-Moody algebra $\mathfrak{g}$ associated to $\widehat{Q}$. It is a natural question whether it is possible to realize the positive part $\mathfrak{n}\subset \mathfrak{g}$ using the semistable irreducible components in different $\Lambda_{\alpha}$'s.

A construction of the Lie algebra $\mathfrak{n}$ using the semistable irreducible components in $\Lambda_{\alpha}$ has been suggested by string theorists \cite{FM}. In the physics setting, the semistable irreducible components in $\Lambda_{\alpha}$ correspond to BPS states on the crepant resolution $\widetilde{\mathbb{C}^2/\Gamma}.$ The algebra structure of the set of BPS states was first suggested in \cite{HM} due to the similarity between certain one-loop integrals and the denominator formulas of the corresponding Kac-Moody algebra. A mathematically rigorous version of the construction of \cite{FM} is conjectured but not proved by Frenkel-Malkin-Vybornov \cite{FMV2}. This construction involves a restriction of the Hall algebra multiplication for $\Pi_{\widehat{Q}}$-mod to semistable elements. In this paper, we confirm the main conjecture of \cite{FMV2}. Our main result can be formulated as follows:

\begin{theorem}  Let $H_{\widehat{Q}}$ denote the subalgebra of the Hall algebra of $\Pi_{\widehat{Q}}$-mod generated by the elements $\{[S_i]| S_i \subset \mbox{Rep}(\Pi_{\widehat{Q}}, \alpha_i), i\in Q_0\}$. Let $H^{ss}_{\widehat{Q}}\subset H_{\widehat{Q}}$ denote the vector subspace consisting of constructible functions supported on semistable irreducible components. Then the restriction of the commutator in $H_{\widehat{Q}}$ to $H^{ss}_{\widehat{Q}}$ defines a Lie bracket. Moreover, $H^{ss}_{\widehat{Q}}$ equipped with this bracket is isomorphic to $\mathfrak{n}.$

\end{theorem}

The isomorphism between $\mathfrak{n}$ and $H^{ss}_{\widehat{Q}}$ complements Lusztig's construction \cite[Section 12]{L}, which defines an isomorphism $\psi:U(\mathfrak{n})\xrightarrow{\sim} H_{\widehat{Q}}.$ Recently, Kontsevich-Soibelman \cite{KS} have defined a version of the Hall algebra related to Lusztig's construction of the canonical basis \cite{L} that is thought to encode the algebra structure of the set of BPS states. Their construction uses cohomology spaces of moduli stacks of objects rather than constructible functions. These constructions are related by the fact that the semistable irreducible components in $\Lambda_{\delta}$ index a basis for the cohomology space $H^2(\widetilde{\mathbb{C}^2/\Gamma})$. 

In general, Hall algebras provide a framework for studying Donaldson-Thomas type invariants, which are thought to ``count" BPS states. The fact that the product in $H^{ss}_{\vec{Q}}$ is given by multiplication in the Hall algebra is important due to the functorial properties of Hall algebras (see e.g. \cite{C}). As shown by Kapranov-Vasserot \cite{KV}, the McKay correspondence lifts to a derived equivalence
\[Rp_{2*}RHom(\mathcal{O}_{\Sigma}, \cdot):D^b(Coh(\widetilde{\mathbb{C}^2/\Gamma}))\rightarrow D^b(Coh_\Gamma(\mathbb{C}^2))\]
Moreover, there is an equivalence $Coh_\Gamma(\mathbb{C}^2)\cong \Pi_{\widehat{Q}}$-mod. Physicists suggest that string theory on the crepant resolution $\widetilde{\mathbb{C}^2/\Gamma}$ should be related to string theory on the Kleinian singularity $\mathbb{C}^2/\Gamma$. Similar predictions have been made relating the Donaldson-Thomas type invariants of CY3 orbifolds and their crepant resolutions \cite{BCY}. Determining the structure of the Hall algebras of categories such as $\Pi_{\widehat{Q}}$-mod should play an important role in the understanding of these conjectures.

In addition to these (conjectural) applications, Theorem 1 also has important implications for Kac's conjecture, which we now discuss. Let $\vec{Q}=(Q_0, Q_1, s, t)$ be any finite quiver without loops. Here $Q_0$ denotes the set of vertices, $Q_1$ denotes the set of arrows, and the maps $s, t: Q_1\rightarrow Q_0$ assign to each arrow their starting and terminal vertices, respectively. Associated to such a quiver $\vec{Q}$ with $\#Q_0=n$ vertices is the $n\times n$ Cartan matrix $(c_{i,j})$ defined by 

\[ c_{i,j} = \left\{ \begin{array}{ll}
         2 & \mbox{if $i=j$}\\
        -\# \{a\in Q_1: (s(a), t(a))=(i,j) \mbox{ or } (j,i)\} & \mbox{if $i\ne j$} \end{array} \right. \] 
One can then associate to $(c_{i,j})$ a Kac-Moody algebra $\mathfrak{g}.$  In \cite{K1}, Kac proved that $\vec{Q}$ has an indecomposable representation in dimension $\alpha\in \mathbb{N}^{Q_0}$ if and only if $\alpha$ is a positive root of $\mathfrak{g}.$ Furthermore, there is a unique isomorphism class of indecomposable representations corresponding to each real root and there are several isomorphism classes corresponding to each imaginary root. When the base field $k$ is algebraically closed, the isomorphism classes of indecomposable representations corresponding to each imaginary root form an infinite family. However, over finite fields one can hope to relate these isomorphism classes to the root space $\mathfrak{g}_{\alpha}.$

Let $k=\mathbb{F}_q$ and let $a_{\alpha}(q)$ denote the number of isomorphism classes of absolutely indecomposable representations of $\vec{Q}$ in dimension $\alpha.$ Recall that a representation is said to be absolutely indecomposable if it is indecomposable over every field extension of the base field $k.$ In \cite{K2}, Kac proved that $a_{\alpha}(q)$ is a polynomial in $q$ with integral coefficients. Moreover, he made the following conjecture:

\begin{conjecture} (\cite{K2}) Let $\vec{Q}$ be a finite quiver without loops and let $\alpha$ be a positive root of the Kac-Moody algebra $\mathfrak{g}$ associated to $\vec{Q}$. Then

a) $a_{\alpha}(q)\in \mathbb{N}[q]$

b) $a_{\alpha}(0)=\dim \mathfrak{g}_{\alpha}$

\end{conjecture}

After decades of little progress, this conjecture was proved for indivisible dimension vectors by Crawley-Boevey and Van den Bergh \cite{CBVdB}. Conjectures 1.1.b and 1.1.a were proved in full generality by Hausel \cite{Ha} and Hausel-Letellier-Rodriguez-Villegas \cite{HLRV}, respectively. Kac's conjecture was known in the case of Dynkin and affine quivers at the time of \cite{K2} due to earlier work \cite{DF, DR, Gab, Na}, however it was not known to hold for any wild quiver until the work of Hausel et al.

Conjecture 1.1.a is suggestive of a cohomological interpretation of the polynomials $a_{\alpha}(q)$. However, the moduli stacks Rep$^{a.i.}(\vec{Q}, \alpha)//GL(\alpha)$ are not representable and any generalization of the Grothendieck-Lefschetz trace formula \cite{Be, Kim} to this setting counts isomorphism classes with multiplicity. Crawley-Boevey and Van den Bergh \cite{CBVdB} bypass this difficulty by relating the absolutely indecomposable representations of $\vec{Q}$ to the semistable representations of the preprojective algebra $\Pi_{Q}=\mathbb{C}\overline{Q}/{\sum [a, a^*]}$ associated to $\vec{Q}$. Explicitly, they prove that
\[a_{\alpha}(q)= \sum_{i=0}^d \dim H^{2d-2i}(X_s, \mathbb{C})q^i\]
when $\alpha$ is an indivisible dimension vector. Here $X_s$ is the moduli space of semistable $\alpha$-dimensional representations of the preprojective algebra $\Pi_Q$ and $d=$ dim $X_s/2.$ It follows from this formula that the constant term $a_{\alpha}(0)$ counts the number of semistable irreducible components in Lusztig's nilpotent variety $\Lambda_{\alpha}$. The authors of \cite{CBVdB} pose the following question:

\begin{question} Can one establish an explicit bijection between the (semi)stable irreducible components in $\Lambda_{\alpha}$ and a basis for $\mathfrak{n}_{\alpha}$?
\end{question}

Theorem 1 gives a positive answer to this question when $\vec{Q}$ is affine. The construction of $H^{ss}_Q$ can be seen to give a conceptual explanation for Kac's constant term conjecture (Conjecture 1.1.b) in the case of an affine quiver. Although $\mathfrak{n}$ has been constructed from the indecomposable representations of an affine quiver \cite{FMV1}, without the cohomological interpretation of $a_{\alpha}(q)$ the connection between $\mathfrak{n}$ and the constant terms $a_{\alpha}(0)$ is ad hoc. The link between the constant terms $a_{\alpha}(0)$ and the semistable irreducible components in $\Lambda_{\alpha}$ is also especially direct in the case of an affine quiver. When the orientation of $\vec{Q}$ is chosen such that each vertex is either a source or a sink, it should be possible to reformulate our construction of $\mathfrak{n}$ entirely in terms of the McKay correspondence (i.e. in terms of representations of $\Gamma$, cf \cite{L2}). We hope to consider this problem in a future paper.

The paper is organized as follows: In section 2, we review the representation theory of affine quivers, including the proof of Conjecture 1.1 in this special case. In section 3, we discuss different moduli spaces of quiver representations and explain the connection between $a_{\alpha}(0)$ and the semistable irreducible components in Lusztig's nilpotent variety $\Lambda_{\alpha}.$ In section 4, we return to the case when $\vec{Q}$ is an affine quiver. We discuss different constructions of $\mathfrak{n}$, and finally construct $\mathfrak{n}$ using the semistable irreducible components in $\Lambda_{\alpha}$ (Theorem 1).

\emph{Acknowledgements.} I am grateful to my advisor, Professor Mikhail
Kapranov, for his guidance and ongoing support. I would also like to thank Professor Igor Frenkel for many interesting and helpful conversations. Finally, I would like to thank Professor Emmanuel Letellier for directing me to \cite{HLRV} for the proof of Conjecture 1.1.a.

\section{Kac's conjecture for affine quivers}

\subsection{Affine root systems and Coxeter transformations}

Let $\vec{Q}=(Q_0, Q_1, s, t)$ be an affine quiver, i.e. a quiver with underlying graph an extended A-D-E Dynkin diagram that has no oriented cycles. The Euler form $e:\mathbb{Z}^{Q_0}\times \mathbb{Z}^{Q_0}\rightarrow \mathbb{Z}$ of $\vec{Q}$ is the biadditive form defined by
\[ e(\alpha_i,\alpha_j) = \left\{ \begin{array}{ll}
         1 & \mbox{if $i=j$}\\
        -\# \{a\in Q_1: (s(a), t(a))=(i,j)\} & \mbox{if $i\ne j$} \end{array} \right. \] 
for $i, j\in Q_0.$ Here $\alpha_i$ and $\alpha_j$ denote the standard basis vectors in $\mathbb{Z}^{Q_0}$ corresponding to $i$ and $j$. The symmetrized Euler form $(, ):\mathbb{Z}^{Q_0}\times \mathbb{Z}^{Q_0}\rightarrow \mathbb{Z}$ is defined by $(\alpha, \beta)=e(\alpha, \beta)+e(\beta, \alpha)$ and the Tits form $q: \mathbb{Z}^{Q_0}\rightarrow \mathbb{Z}$ is defined by $q(\alpha)= e(\alpha, \alpha)$. Note that the entries of the Cartan matrix associated to $\vec{Q}$ are given by $c_{i, j}=(\alpha_i, \alpha_j)$. Vectors $\alpha\in \mathbb{Z}^{Q_0}$ satisfying $q(\alpha)\le 1$ are referred to as \textit{roots} and the standard basis vectors $\alpha_i\in \mathbb{Z}^{Q_0}, i\in Q_0$ are referred to as \textit{simple roots}. Roots which belong to $\mathbb{N}^{Q_0}$ are referred to as \textit{positive roots}. Denote the set of roots by $\Delta(\vec{Q})$ and the set of positive roots by $\Delta^+(\vec{Q})$. To each simple root there is an associated automorphism $s_i\in \mbox{Aut} (\mathbb{Z}^{Q_0})$ given by
\[s_i(\beta)= \beta - (\beta, \alpha_i)\alpha_i\]
for $\beta\in \mathbb{Z}^{Q_0}.$ Because $\vec{Q}$ is assumed to have no loops, each automorphism $s_i$ has order 2 and we refer to these as \textit{simple reflections}. The roots obtained from simple roots by simple reflections are referred to as \textit{real roots} and all other roots are referred to as \textit{imaginary roots}. Denote the set of real and imaginary roots by $\Delta^{re}(\vec{Q})$ and $\Delta^{im}(\vec{Q})$. Because $\vec{Q}$ is affine, there is a unique imaginary root $\delta$ such that $rad(q)=\mathbb{Z}\delta$ and such that all imaginary roots are of the form $m\delta$ for $m\in \mathbb{Z}$. A vertex $e\in Q_0$ is called an \textit{extending vertex} if $\delta_e=1$. Any affine quiver $\vec{Q}$ contains an extending vertex, which in general is not unique. By removing an extending vertex from $\vec{Q}$, one obtains a (non-extended) Dynkin quiver of the same type as $\vec{Q}$. The \textit{affine Coxeter transformation} $c:\mathbb{Z}^{Q_0}\rightarrow \mathbb{Z}^{Q_0}$ is a product of all simple reflections that satisfies $e(\alpha, \beta) = - e(\beta, c \alpha)$ for all $\alpha, \beta\in \mathbb{Z}^{Q_0}.$ Note that $c\alpha=\alpha$ if and only if $\alpha\in \mathbb{Z}\delta.$ Roots for which the affine Coxeter transformation $c$ acts periodically are referred to as \textit{regular roots} $\Delta^{reg}(\vec{Q})$. All other roots are referred to as \textit{irregular roots} $\Delta^{irr}(\vec{Q})$ and these are related to simple roots $\alpha_i$ by repeated application of $c$. Note that $\Delta^{im}(\vec{Q})\subset \Delta^{reg}(\vec{Q})$ but real roots may be regular or irregular, depending on the orientation of $\vec{Q}$. A criterion for distinguishing between regular and irregular roots can be given using the \textit{defect} $\partial:\mathbb{Z}^{Q_0}\rightarrow \mathbb{Z}$. This is defined by
\[\partial(\alpha)=e(\delta, \alpha)\]
A root $\alpha\in \Delta(\vec{Q})$ is regular if and only if $\partial(\alpha)=0$.

Associated to $\vec{Q}$ is the affine Kac-Moody algebra $\mathfrak{g}$. This is an infinite-dimensional $\mathbb{Z}^{Q_0}$-graded Lie algebra. For a given $\alpha\in \mathbb{Z}^{Q_0}$, the subspace $\mathfrak{g}_{\alpha}$ is non-empty if and only if $\alpha=0$ or $\alpha$ is a root. When $\alpha$ is a root, the dimension of $\mathfrak{g}_{\alpha}$ is called the multiplicity of $\alpha$. Real roots have multiplicity $1$ and imaginary roots have multiplicity $n$ when $\vec{Q}$ has $n+1$ vertices. The vector space
\[\mathfrak{n}=\bigoplus_{\alpha\in \Delta^+(\vec{Q})} \mathfrak{g}_{\alpha}\]
forms a Lie subalgebra of $\mathfrak{g}$. It can be constructed as the quotient of the free Lie algebra on the generators $\{e_i\}_{i\in Q_0}$ over $\mathbb{C}$ by the \textit{Serre relations}
\[(ad(e_i))^{1-c_{i,j}}(e_j)=0\]
for $i\ne j$.

\subsection{Representation theory of affine quivers}

Recall that the category of representations of $\vec{Q}$ over $k$ is equivalent to the category of $k\vec{Q}$-modules, where $k\vec{Q}$ is the path algebra of $\vec{Q}.$ Let $k\vec{Q}$-mod denote the category of left $k\vec{Q}$-modules and let mod-$k\vec{Q}$ denote the category of right $k\vec{Q}$-modules. The duality functors $D:k\vec{Q}\mbox{-mod} \rightarrow \mbox{mod-}k\vec{Q}$ and $D: \mbox{mod-}k\vec{Q} \rightarrow k\vec{Q}\mbox{-mod}$ are defined by $D(X)=\mbox{Hom}(X,k)$. To each vertex $i\in Q_0$ we can associate a simple representation $S_i$ given by $(S_i)_i=k$ and $(S_i)_j=0$ for $j\ne i$. Because $\vec{Q}$ is assumed to have acyclic orientation, the set $\{S_i\}_{i\in Q_0}$ is the full set of simple representations of $\vec{Q}$. Therefore, we can identify the Grothendieck group $K_0(k\vec{Q}\mbox{-mod})$ with $\mathbb{Z}^{Q_0}$ and to each representation $X\in k\vec{Q}$-mod we can associate a \textit{dimension vector} $\underline{dim}(X)\in \mathbb{Z}^{Q_0}$. Note that
\[\mbox{dim Hom}(S_i, S_j)=\delta_{i, j}\]
and
\[\mbox{dim Ext}^1(S_i, S_j)= \#\{a\in Q_1: (s(a), t(a))=(i, j)\}\]
It follows that 
\[\mbox{dim Hom}(S_i, S_j)-\mbox{dim Ext}^1(S_i, S_j)= e(\underline{dim}(S_i), \underline{dim}(S_j))\]
for any $i, j\in Q_0$. In fact, using dimension shifting, one can show that this holds in general:
\[\mbox{dim Hom}(M, N) - \mbox{dim Ext}^1(M, N)= e(\underline{dim}(M), \underline{dim}(N))\]
for any $M, N\in k\vec{Q}$-mod. In addition to the representations $\{S_i\}$, also associated to each vertex $i\in Q_0$ is the projective (respectively, injective) indecomposable $P(i)$ (respectively, $I(i)$). This has as a basis the set of paths starting (respectively, terminating) at $i$. These form the full set of isomorphism classes of projective (respectively, injective) indecomposable representations of $\vec{Q}$.

\begin{definition} The Auslander-Reiten translation $\tau:k\vec{Q}\mbox{-mod}\rightarrow k\vec{Q}\mbox{-mod}$ is defined by $\tau(X)=\mbox{DExt}^1(X, k\vec{Q}).$ The inverse Auslander-Reiten translation $\tau^{-}:k\vec{Q}\mbox{-mod}\rightarrow k\vec{Q}\mbox{-mod}$ is defined by $\tau^{-}(X)=\mbox{Ext}^1(D(X), k\vec{Q}).$

\end{definition}

The Auslander-Reiten translation defines an equivalence from non-projective indecomposables to non-injective indecomposables and the inverse Auslander-Reiten translation provides its quasi-inverse. Using the fact that $K_0(k\vec{Q}\mbox{-mod})\cong \mathbb{Z}^{Q_0}$, $\tau$ (together with the Nakayama functor, which we have not defined) can be seen to give a categorification of the affine Coxeter transformation $c.$ For non-projective indecomposables $X\in k\vec{Q}$-mod, the Auslander-Reiten translation $\tau$ satisfies $\underline{dim}(\tau(X))=c(\underline{dim}(X))$. It is known that $\tau$ and $\tau^{-}$ can be expressed as compositions of the Bernstein-Gelfand-Ponomarev reflection functors \cite{BGP}, which categorify the simple reflections $s_i$.

\begin{definition} An indecomposable representation $X\in k\vec{Q}$-mod is said to be \textit{preprojective} if $\tau^N(X)=0$ for some $N$ and \textit{preinjective} if $(\tau^{-})^N(X)=0$ for some $N$. An indecomposable representation that is not preprojective or preinjective is referred to as \textit{regular}.

\end{definition}
Note that any projective indecomposable $X$ is preprojective, as $\tau(X)=0$. Similarly, any injective indecomposable is preinjective. Decomposable representations of $\vec{Q}$ are said to be \textit{preprojective}, \textit{preinjective}, or \textit{regular} if each of their summands is a preprojective, preinjective, or regular indecomposable, respectively. Denote the full subcategories of preprojective, preinjective, and regular modules by $\mathcal{P}, \mathcal{I}$, and $\mathcal{R}$. For an abelian category $\mathcal{A}$ with full subcategories $\mathcal{B}, \mathcal{C}\subset \mathcal{A}$, we write
\[\mathcal{B}\prec \mathcal{C}\]
if Hom$(Y,X)=\mbox{Ext}^1(X,Y)=0$ whenever $X\in \mathcal{B}$ and $Y\in \mathcal{C}$. Using basic properties of $\tau$ and $\tau^{-}$, one can show that
\[\mathcal{P} \prec \mathcal{R} \prec \mathcal{I}\]
It follows that the category of regular modules is closed under extensions.

\begin{proposition} A representation $X$ is regular if and only if $\partial(\underline{dim}(X))=0$ and $\partial(\underline{dim}(Y))\le 0$ for all $Y\subset X$.

\end{proposition}

\begin{proof}
If $M$ is a regular indecomposable, then $\underline{dim}(M)$ must be a regular root. Hence, $\partial(\underline{dim}(M))=0$. By the additivity of the Euler form, $\partial(\underline{dim}(X))=0$ for any regular module $X$. If $Y$ is a submodule of $X$, then $Y$ must be either preprojective or regular. If $Y$ is preprojective, then
\[\partial(\underline{dim}(Y))=e(\delta, \underline{dim}(Y))=-e(\underline{dim}(Y), \delta)=-(\mbox{dim Hom}(Y, Z) - \mbox{dim Ext}^1(Y, Z))\]
for some regular indecomposable $Z\in \mbox{Rep}(\vec{Q}, \delta)$. Assuming that $Y$ is a submodule of a regular module, it is possible to choose $Z$ such that Hom$(Y, Z)\ne 0$ (in fact, this holds for any preprojective indecomposable $Y$). Because $\mathcal{P}\prec \mathcal{I}$, we have Ext$^1(Y, Z)=0$ and hence, $\partial(\underline{dim}(Y))<0$. One can show by similar arguments that for any preinjective indecomposable $W$, $\partial(\underline{dim}(W))>0$.

Now assume that a representation $X$ satisfies $\partial(\underline{dim}(X))=0$ and $\partial(\underline{dim}(Y))\le 0$ for all $Y\subset X$. Then it cannot admit any preinjective submodules. Moreover, by the additivity of $\partial$, it cannot admit any preprojective summands. Therefore, it must be regular.

\end{proof}

Preprojective and preinjective indecomposables can be obtained from simple representations by repeated applications of $\tau$ and $\tau^{-}$ and they are relatively well understood. Regular indecomposable modules are generally difficult to classify, however in the case of an affine quiver they have been described by Dlab-Ringel \cite{DR}. Specifically, it is shown that the set of regular indecomposable modules breaks up into a disjoint union of \textit{tubes}, which are uniserial subcategories that are closed under extensions. The set of tubes is indexed by the projective line $\mathbb{P}^1(k)$ and there are no non-trivial morphisms or extensions between different tubes. Moreover, each tube is stable under the Auslander-Reiten translation $\tau$. If $\tau(X)\cong X$ for all objects $X$ in a tube, then it is said to be \textit{homogeneous}. Otherwise, it is called \textit{non-homogeneous} and $\tau^{N}(X)\cong X$ for all objects $X$ in the tube for some period $N$. At most 3 tubes are non-homogeneous and the set of non-homogenous tubes together with their periods is called the \textit{tubular type} of $\vec{Q}$. Considered as categories, each tube contains a set of simple objects that are referred to as \textit{regular simple} objects. Note that these are not simple objects in $k\vec{Q}$-mod, however over an algebraically closed field $k$ they are \textit{bricks}, i.e. they satisfy End$(X)\cong k$. In homogeneous tubes, there is a unique regular simple object in dimension $\delta$ up to isomorphism. For non-homogeneous tubes, there are $N$ different isomorphism classes of regular simple objects whose dimension vectors correspond to real roots $\alpha < \delta$. Here $N$ is the period of $\tau$ acting on the tube. The functor $\tau$ acts transitively on the isomorphism classes of regular simple objects in a non-homogeneous tube and for any regular simple object $X$ there is an equality
\[\underline{dim}(X)+\underline{dim}(\tau(X))+\underline{dim}(\tau^2(X))+\cdots + \underline{dim}(\tau^{N-1}(X))=\delta\]
The fact that each tube is uniserial means that each regular indecomposable object has a unique composition series consisting of regular simple objects. Furthermore, it is known that each regular indecomposable is uniquely determined by the length of its composition series and its top. A non-homogeneous tube with period $N$ contains $N$ isomorphism classes of regular indecomposables in dimension $m\delta$ for any $m\in \mathbb{N}$.

The results of \cite{DR} can be formulated as proving the existence of a collection of fully faithful functors $\mathcal{C}_z:\mbox{Rep}^{nil}(C_{N_z})\rightarrow \mbox{Rep}(\vec{Q})$ parameterized by $z\in \mathbb{P}^1(k)$. Here $C_{N_z}$ is either the Jordan quiver for $N_z=1$ or the cyclic quiver of length $N_z$ for $N_z\ge 2$. Note that we only consider nilpotent representations of $C_{N_z}$, i.e. those for which there exists some $M$ such that the composition of any path of length $\ge M$ is zero. Equivalently, we consider representations which have composition series consisting of simple representations associated to the vertices of $C_{N_z}$. The image of each functor $\mathcal{C}_z$ is the tube indexed by $z\in \mathbb{P}^1(k)$. When $N_z=1$ the corresponding tube is homogeneous and when $N_z\ge 2$ it is non-homogeneous with period $N_z$. The images of the simple objects in $C_{N_z}$ are the corresponding tube's regular simple objects. Moreover, if the dimension of a representation $X\in \mbox{Rep}^{nil}(C_{N_z})$ is the first imaginary root of $C_{N_z}$, then the dimension of $\mathcal{C}_z(X)$ is the first imaginary root of $\vec{Q}$.

The parameterization of the tubes in Rep$(\vec{Q})$ by $\mathbb{P}^1(k)$ is fixed by a choice of extending vertex $e$. Explicitly, each tube contains a unique (up to isomorphism) indecomposable $X$ in dimension $\delta$ such that the top of $X$ is supported on $e$. These indecomposables are then parameterized by maps from the projective indecomposable $P(e)$ to a fixed indecomposable representation in dimension $\underline{dim}(P(e))+\delta$. This parameterization can be formulated in terms of a fully faithful functor $\mathcal{K}:\mbox{Rep}(K)\rightarrow \mbox{Rep}(\vec{Q})$, where $K$ is the Kronecker quiver. The functor $\mathcal{K}$ takes the simple object $P(1)\in \mbox{Rep}(K)$ to the projective indecomposable $P(e)\in \mbox{Rep}(\vec{Q})$.

For $\alpha$ an irregular root, let $P_{\alpha}$ denote the unique isomorphism class of indecomposables in dimension $\alpha$. For a non-homogeneous tube of period $N_z$, we can label the regular simple objects by elements $j\in \mathbb{Z}/N_z\mathbb{Z}$. Let $\{P_{z, j, 1}\}_{j\in \mathbb{Z}/N_z\mathbb{Z}}$ denote the isomorphism classes of regular simple objects in such a tube and let $P_{z, j, l}$ denote the unique isomorphism class of regular indecomposables with length $l$ and top isomorphic to $P_{z, j, 1}$. For a homogeneous tube indexed by $z$, let $P_{z, 1, l}$ denote the unique isomorphism class of indecomposables of length $l$.

\begin{proposition} (\cite{DR}) Let $\vec{Q}$ be an affine quiver with $n+1$ vertices. Then the set
\[\{P_{\alpha}\}_{\alpha\in \Delta^{irr}}\cup \{P_{z, j, l}\}_{z\in \mathbb{P}^1(k)}\]
is the full set of isomorphism classes of indecomposable representations of $\vec{Q}$. In particular, the set
\[\{P_{z, j, mN_z}\}_{z\in \mathbb{P}^1(k)}\]
is the full set of isomorphism classes of indecomposable representations in dimension $m\delta$. The periods $N_z$ satisfy the following identity:

\[\sum_{z\in \mathbb{P}^1(k)} (N_z-1) = n - 1\]

\end{proposition}

Setting $k=\mathbb{F}_q$, one obtains

\[\sum_{z\in \mathbb{P}^1(k)} N_z = q+n = q+\mbox{dim } \mathfrak{g}_{\delta}\]
From the statements above, we know that this is the total number of indecomposable representations in dimension $\delta$. Because $\delta$ is indivisible, this is also the number of absolutely indecomposable representations in this dimension and hence $a_{\delta}(q)= \mbox{dim } \mathfrak{g}_{\delta}+q$. This confirms Kac's conjecture for dimension $\delta$. There is a bijection between indecomposable (and hence, absolutely indecomposable) representations in dimension $m\delta$ and indecomposable representations in dimension $\delta$. Therefore, $a_{m\delta}(q)=a_{\delta}(q)$. Because dim $\mathfrak{g}_{\delta}=$ dim $\mathfrak{g}_{m\delta}$, Kac's conjecture follows in this case as well.

\section{Lusztig's nilpotent variety and $a_{\alpha}(0)$}

\subsection{Moduli of representations}

Let $k$ be an algebraically closed field. It is known that in this case any finite-dimensional basic $k$-algebra is isomorphic to a quotient $k\vec{Q}/I$ for some two-sided ideal $I\subset k\vec{Q}$. The pair $(\vec{Q},I)$ is called a \textit{quiver with relations}. Note that $K_0(\mbox{Rep}(k\vec{Q}))=K_0(\mbox{Rep}(k\vec{Q}/I))=\mathbb{Z}^{Q_0}$. For $\alpha\in \mathbb{N}^{Q_0}$, one can consider the affine variety

\[\mbox{Rep}(\vec{Q}, \alpha)=\bigoplus_{a\in {Q_1}} \mbox{Hom}(k^{s(a)}, k^{t(a)})\]
The group $GL(\alpha)=\Pi_{i\in {Q_0}} GL(\alpha_i)$ acts on Rep$(\vec{Q}, \alpha)$ by simultaneous conjugation. If $A\cong k\vec{Q}/I$, then we denote the subvariety of Rep$(\vec{Q}, \alpha)$ corresponding to representations of $A$ by Rep$(A, \alpha)$. The action of $GL(\alpha)$ preserves Rep$(A, \alpha)$ and its orbits correspond to isomorphism classes of $\alpha$-dimensional representations of $A$. The affine quotient Rep$(A, \alpha)/GL(\alpha):=\mbox{Spec}(k[\mbox{Rep}(A, \alpha)]^{GL(\alpha)})$ parameterizes closed orbits, which correspond to isomorphism classes of semisimple $\alpha$-dimensional representations. The quotient map Rep$(A, \alpha)\rightarrow \mbox{Rep}(A, \alpha)/ GL(\alpha)$ is a principal $GL(\alpha)$-bundle and it identifies representations with the same Jordan-H\"{o}lder series.

Following King \cite{Ki}, one can also consider the geometric invariant theory quotients of Rep$(A, \alpha).$ The only choice of line bundle in this case is the trivial bundle Rep$(A, \alpha)\times k$, however one can twist the action of $GL(\alpha)$ by a character to define different linearisations. Any character of $GL(\alpha)$ is determined by a weight $\Theta\in \mathbb{Z}^{Q_0}$ and is of the form
\[\chi(g)= \Pi_{i\in Q_0} (\mbox{det} (g_i))^{\Theta_i}\]
For a character $\chi:GL(\alpha)\rightarrow k^*,$ the twisted action of $GL(\alpha)$ on the line bundle Rep$(A, \alpha)\times k$ is given by $g\cdot (x, z) = (g\cdot x, \chi^{-1} (g) \cdot z)$. The geometric invariant theory quotient of Rep$(A, \alpha)$ with respect to this linearisation is defined by
\[\mbox{Rep}(A, \alpha)/_{\chi} GL(\alpha)= \mbox{Proj} (\oplus_{n\ge 0} k[\mbox{Rep}(A, \alpha)]^{(GL(\alpha), \chi^n)})\]
where $k[\mbox{Rep}(A, \alpha)]^{(GL(\alpha), \chi^n)} = \{ f\in k[\mbox{Rep}(A, \alpha)] | f(g\cdot x)=\chi^n(g)f(x) \mbox{ } \forall g\in GL(\alpha)\}.$

A representation $V\in \mbox{Rep}(A, \alpha)$ is said to be \textit{$\chi$-semistable} if there exists a semi-invariant polynomial $f\in k[\mbox{Rep}(A, \alpha)]^{(GL(\alpha), \chi^n)}$ for some $n\ge 0$ such that $f(V)\ne 0$. A $\chi$-semistable representation is said to be \textit{$\chi$-stable} if in addition the orbit $GL(\alpha)\cdot W$ is closed and is of dimension dim$(GL(\alpha))-1$. The set of $\chi$-semistable representations forms an open subset Rep$^{ss}(A, \alpha)$ of Rep$(A, \alpha)$ and the set of $\chi$-stable representations Rep$^s(A, \alpha)$ forms an open subset of Rep$^{ss}(A, \alpha)$. The quotient Rep$(A, \alpha)/_{\chi} GL(\alpha)$ parameterizes closed orbits in Rep$^{ss}(A, \alpha)$. The quotient map Rep$^{ss}(A, \alpha)\rightarrow \mbox{Rep}(A, \alpha)/_{\chi} GL(\alpha)$ is again a principal $GL(\alpha)$-bundle and it identifies points whose orbit closures intersect.

Now let $\Theta\in \mathbb{Z}^{Q_0}$ be a weight and let $\chi$ be the associated character. Define a bilinear form $\langle, \rangle$ on $\mathbb{Z}^{Q_0}$ by $\langle \alpha_i, \alpha_j \rangle = \delta_{ij}$ for $i, j\in Q_0.$ We say that a representation $V\in \mbox{Rep}(\vec{Q}, \alpha)$ is $\Theta$-\textit{semistable} (resp., $\Theta$-\textit{stable}) if $\langle \Theta, \underline{dim} (V') \rangle \le 0$ (resp. $\langle \Theta, \underline{dim}(V') \rangle <0$) for all proper subrepresentations $V'\subset V$. King has proved the following theorem:

\begin{theorem} (\cite{Ki}) A representation $V\in \mbox{Rep}(A, \alpha)$ is $\chi$-semistable (resp., $\chi$-stable) if and only if it is $\Theta$-semistable (resp., $\Theta$-stable).

\end{theorem}

We refer to a weight $\Theta\in \mathbb{Z}^{Q_0}=K_0(\mbox{Rep}(A))$ as a \textit{stability condition} on the category Rep$(A)$. Note that the choice $\Theta\in \mathbb{Z}^{Q_0}$ is equivalent to a choice of homomorphisms $\lambda:K_0(\mbox{Rep}(A))\rightarrow \mathbb{Z}$. As is clear from the theorem, different stability conditions determine different stable and semistable representations. However, we drop the prefix $\Theta$- when there is no confusion. It follows from the theorem that all simple representations are stable and all stable representations are indecomposable. For a fixed stability condition, the collection of semistable representations forms a full abelian subcategory of Rep$(A)$. The simple objects in this subcategory are the stable objects and each semistable object has a Jordan-H\"{o}lder filtration with stable composition factors. Two semistable representations are said to be \textit{S-equivalent} if their composition factors are the same. A semistable representation is said to be \textit{polystable} if it is a direct sum of stable representations. The polystable representations give the closed orbits in Rep$^{ss}(A, \alpha)$ and every semistable representation is S-equivalent to a polystable one. The quotient Rep$^{ss}(A, \alpha)/_{\chi} GL(\alpha)$ forms a coarse moduli space for families of semistable representations up to S-equivalence. Moreover, the open subset Rep$^s(A, \alpha)/_{\chi} GL(\alpha)$ forms a coarse moduli space for families of stable representations up to isomorphism. When the dimension vector $\alpha$ is indivisible, Rep$^s(A, \alpha)/_{\chi} GL(\alpha)$ is a fine moduli space.

The inclusion 
\[k[\mbox{Rep}(A, \alpha)]^{GL(\alpha)}\hookrightarrow \bigoplus_{n\ge 0} k[\mbox{Rep}(A, \alpha)]^{(GL(\alpha), \chi^n)}\]
induces a projective map 
\[\pi: \mbox{Rep}(A, \alpha)/_{\chi} GL(\alpha)\rightarrow \mbox{Rep}(A, \alpha)/GL(\alpha)\] 

In the case that $\vec{Q}$ has no oriented cycles, 
\[k[\mbox{Rep}(\vec{Q}, \alpha)]^{GL(\alpha)}=k\]
and the quotient Rep$(\vec{Q}, \alpha)/_{\chi} GL(\alpha)$ is projective. In general, the subscheme $\pi^{-1}(0)$ parameterizes semistable representations that are nilpotent. For $k=\mathbb{C},$ there is a homotopy equivalence between Rep$(A, \alpha)/_{\chi} GL(\alpha)$ and $\pi^{-1}(0).$

\subsection{Preprojective algebras and $a_{\alpha}(q)$}

Now consider the double quiver $\overline{Q}$. This is the quiver obtained from $\vec{Q}$ by adding a reverse arrow $a^*:j\rightarrow i$ for each arrow $a:i\rightarrow j\in Q_1.$ The affine variety Rep$(\overline{Q}, \alpha)$ is equipped with a symplectic form

\[\langle x, y \rangle = \sum_{a\in Q_1} tr [x_{a}, y_{a^*}]\]
and it can be viewed as the cotangent space of Rep$(\vec{Q}, \alpha)$. The moment map associated to the action of $GL(\alpha)$ is the map $\mu:\mbox{Rep}(\overline{Q}, \alpha)\rightarrow \Pi_{i\in Q_0} \mathfrak {gl}(\alpha_i)$ given by
\[\mu_i(x) = \sum_{a\in Q_1, s(a)=i} x_ax_{a^*} - \sum_{a\in Q_1, t(a)=i} x_{a^*}x_a \in \mathfrak{gl}(\alpha_i)\]
 for $i\in Q_0$. Associated to $\vec{Q}$ is the \textit{preprojective algebra} $\Pi_Q$ defined by

\[\Pi_Q=k\overline{Q}/\{\sum_{a\in Q_1, s(a)=i} f_af_{a^*} - \sum_{a\in Q_1, t(a)=i} f_{a^*}f_a\}\]
For any $\alpha\in \mathbb{N}^{Q_0}$, the representation space Rep$(\Pi_Q, \alpha)$ is the set $\mu^{-1}(0)\subset \mbox{Rep}(\overline{Q}, \alpha)$.
Let $\Lambda_{\alpha}$ denote the space of nilpotent representations in Rep$(\Pi_Q, \alpha)$. This is known as Lusztig's nilpotent variety. Lusztig has shown that $\Lambda_{\alpha}$ is a Lagrangian subvariety of the symplectic vector space Rep$(\overline{Q}, \alpha)$.

Given an element $\lambda\in \mathbb{Z}^{Q_0}$ one can also define the \textit{deformed preprojective algebra} $\Pi^{\lambda}_Q$ by
\[\Pi^{\lambda}_Q=k\overline{Q}/\{\sum_{a\in Q_1, s(a)=i} f_af_{a^*} - \sum_{a\in Q_1, t(a)=i} f_{a^*}f_a - \sum_{i\in Q_0} \lambda_i e_i\}\]
The weight $\lambda$ defines an element $(\lambda_i I_{\alpha_i}) \in \Pi_{i\in Q_0} \mathfrak{gl_i}(\alpha_i)$ and the representation space Rep$(\Pi^{\lambda}_Q, \alpha)$ can identified with the set $\mu^{-1}(\lambda)$. Moreover, $\lambda$ defines a stability condition on both $\Pi^{\lambda}_Q$-mod and $\Pi_Q$-mod. We say that $\lambda$ is \textit{generic} with respect to $\alpha\in \mathbb{N}^{Q_0}$ if $\langle \lambda, \alpha \rangle =0$ and $\langle \lambda, \beta \rangle \ne 0$ for all $0<\beta < \alpha.$ Such a stability condition always exists when $\alpha$ is indivisible. It is clear that no representations are strictly semistable with respect to a generic stability condition. Crawley-Boevey and Van den Bergh have proved the following:

\begin{theorem} (\cite{CBVdB}) Let $\alpha$ be indivisible and choose $\lambda$ to be generic with respect to $\alpha$. Set $k=\overline{\mathbb{F}_q}$ for $q=p^t$ and let $X_s=Rep(\Pi_Q, \alpha)/_{\lambda} GL(\alpha)$ and $X_{\lambda}=Rep(\Pi^{\lambda}_Q, \alpha)/_{\lambda} GL(\alpha)$. Then for $t\gg 0,$ we have

a) $a_{\alpha}(q) = q^{(\alpha, \alpha)-1} X_{\lambda}(\mathbb{F}_q)$

b) $X_{\lambda}(\mathbb{F}_q) = X_s(\mathbb{F}_q)$

\end{theorem}

Using the Grothendieck-Lefschetz formula and the Artin comparison theorem, one can then obtain Crawley-Boevey and Van den Bergh's formula for $a_{\alpha}(q)$. Over $k=\mathbb{C}$, we have a homotopy equivalence between $X_s$ and $\Lambda_{\alpha}/_{\lambda} GL(\alpha)$. It follows that the constant term $a_{\alpha}(0)$ counts the number of top-dimensional components in $\Lambda_{\alpha}/_{\lambda} GL(\alpha)$. Because $\Lambda_{\alpha}$ is Lagrangian and the map $\Lambda_{\alpha} \rightarrow \Lambda_{\alpha}/_{\lambda} GL(\alpha)$ is a principal $GL(\alpha)$-bundle, this is equal to the number of $\lambda$-semistable irreducible components in $\Lambda_{\alpha}$. Using slope semistability \cite{Re, Ru} and the Harder-Narasimhan filtration, together with the fact that $|Irr  \Lambda_{\alpha}|=\dim U(\mathfrak{n})_{\alpha}$ \cite{KaS, L3}, one can show that the number of $\lambda$-semistable irreducible components in $\Lambda_{\alpha}$ is equal to $\dim \mathfrak{n}_{\alpha}$.

\section{The algebra of BPS states}

\subsection{The Frenkel-Kac cocycle}

Now we specialize again to the case of an affine quiver. Let $k=\mathbb{C}$ and let $\Gamma$ be a finite subgroup of $SL(2,\mathbb{C})$. Associated to $\Gamma$ is an affine Dynkin diagram $\widehat{Q}$, as well as the ordinary Dynkin diagram $Q$ obtained by removing an extending vertex from $\widehat{Q}$. The affine quotient Rep$(\Pi_{\widehat{Q}}, \delta)/GL(\delta)$ is isomorphic to the Kleinian singularity $\mathbb{C}^2/\Gamma$. For generic stability conditions, the projective maps
\[\pi: \mbox{Rep}(\Pi_{\widehat{Q}}, \delta)/_{\chi} GL(\delta)\rightarrow \mbox{Rep}(\Pi_{\widehat{Q}}, \delta)/GL(\delta)\]
are crepant resolutions and the fiber $\pi^{-1}(0)$ is isomorphic to $\Lambda_{\delta}/_{\chi} GL(\delta)$. When $\chi$ is not generic, the moduli space Rep$(\Pi_{\widehat{Q}}, \delta)/_{\chi} GL(\delta)$ is singular and $\pi$ gives a partial resolution of $\mathbb{C}^2/\Gamma$. As described in the introduction, our goal will be to give the semistable irreducible components in $\Lambda_{\alpha}$ for various $\alpha\in \Delta^+(\widehat{Q})$ a Lie algebra structure isomorphic to $\mathfrak{n}$. We recall a presentation of $\mathfrak{n}$ due to Frenkel-Kac \cite{FK} that is more amenable to this question than the standard presentation using Serre relations.

Let $\vec{Q}=(Q_0, Q_1, s, t)$ be an affine quiver. For each positive real root $\alpha$, let $\mathfrak{n}_{\alpha}^{\epsilon}(\vec{Q})=\mathbb{C}\tilde{e}_{\alpha}$ be a one-dimensional vector space, and for each positive imaginary root $m\delta$, let $\mathfrak{n}_{m\delta}^{\epsilon}(\vec{Q})=\mathbb{C}[Q_0]/\mathbb{C}\delta$. For an element $\alpha\in \mathbb{C}[\vec{Q}_0]$, we let $\alpha(m)$ denote the corresponding element in $\mathfrak{n}^{\epsilon}_{m\delta}(\vec{Q})$. Define $\epsilon:\mathbb{Z}^{Q_0}\times \mathbb{Z}^{Q_0}\rightarrow \{\pm 1\}$ by
\[\epsilon(\alpha, \beta)=(-1)^{e(\alpha, \beta)}\]
The graded vector space
\[\mathfrak{n}^{\epsilon}(\vec{Q})=\bigoplus_{\alpha\in \Delta^+(\vec{Q})} \mathfrak{n}^{\epsilon}_{\alpha}(\vec{Q})\]
can then be given the structure of a Lie algebra by the following rules:
\[[\tilde{e}_{\alpha}, \tilde{e}_{\beta}]= \left\{ \begin{array}{lcr}
         \epsilon(\alpha, \beta)\tilde{e}_{\alpha+\beta} & \mbox{if $\alpha+\beta\in \Delta^{re}(\vec{Q})$}\\
        \epsilon(\alpha, \beta)\alpha(m) & \mbox{if $\alpha+\beta=m\delta$} \\
        0 & \mbox{ if } \alpha+\beta\notin \Delta^+(\vec{Q}) \end{array} \right. \]

\[[\alpha(m), \tilde{e}_{\beta}]=\epsilon(\alpha, \beta)e(\alpha, \beta)\tilde{e}_{\beta+m\delta}\]
\[[\alpha(m), \beta(n)]=0\]

Now let $\mathfrak{n}$ denote the positive part of the affine Lie algebra $\mathfrak{g}$ associated to $\vec{Q}$.

\begin{theorem} (\cite{FK}) The assignment
\[e_i\mapsto \tilde{e}_{\alpha_i}\]
extends to an isomorphism $\mathfrak{n}\xrightarrow{\sim} \mathfrak{n}^{\epsilon}(\vec{Q})$
\end{theorem}

For any root $\alpha\in \mathbb{Z}^{Q_0}$, define
\[\xi(\alpha)=(-1)^{1+dim End(P)}\]
where $P$ is any indecomposable representation of $\vec{Q}$ in dimension $\alpha$. This is well-defined because each indecomposable $P$ in dimension $m\delta$ satisfies dim End$(P)=m$. Extend $\xi$ to all of $\mathbb{Z}^{Q_0}$ arbitrarily. One can twist $\epsilon$ to obtain a new cocycle $\epsilon^*$ by defining
\[\epsilon^*(\alpha, \beta)= \epsilon(\alpha, \beta)\xi(\alpha+\beta)\xi^{-1}(\alpha)\xi^{-1}(\beta)\]
The two cocycles $\epsilon$ and $\epsilon^*$ correspond to the same element in $H^2(\mathbb{Z}^{Q_0}, \mathbb{Z}/2\mathbb{Z})$. The Lie algebra $\mathfrak{n}^{\epsilon^*}(\vec{Q})$ obtained using $\epsilon^*$ in place of $\epsilon$ is isomorphic to $\mathfrak{n}^{\epsilon}(\vec{Q})$ via the map

\begin{align} 
\mathfrak{n}^{\epsilon}(\vec{Q})&\rightarrow \mathfrak{n}^{\epsilon^*}(\vec{Q})\nonumber \\
\tilde{e}_{\alpha}&\mapsto \xi(\alpha) \tilde{e}_{\alpha} \nonumber \\
\alpha(m)&\mapsto \xi(m\delta) \alpha(m) \nonumber
\end{align}

\subsection{The Hall algebras of $\mathbb{C}\vec{Q}$-mod and $\Pi_Q$-mod}

Let $\vec{Q}$ be an affine quiver. Define
\[\mathcal{L}_{\alpha}(\vec{Q})=M_{GL(\alpha)}(\mbox{Rep}(\vec{Q}, \alpha))\]
to be the set of $GL(\alpha)$-invariant, $\mathbb{C}$-valued constructible functions on Rep$(\vec{Q}, \alpha)$. One can define a multiplication on the graded vector space
\[\mathcal{L}(\vec{Q})=\bigoplus_{\alpha\in \mathbb{Z}^{\vec{Q}_0}} \mathcal{L}_{\alpha}(\vec{Q})\]
using a certain convolution product. For $f_{\alpha}\in \mathcal{L}_{\alpha}(\vec{Q})$ and $f_{\beta}\in \mathcal{L}_{\beta}(\vec{Q})$, define a constructible function 
\[f_{\alpha}\times f_{\beta}:\mbox{Rep}(\vec{Q}, \alpha)\times \mbox{Rep}(\vec{Q}, \beta)\rightarrow \mathbb{C}\]
by the rule
\[(f_{\alpha}\times f_{\beta})(x_1, x_2) = f_{\alpha}(x_1)f_{\beta}(x_2)\]
Let $Corr_1$ be the variety consisting of all pairs of representations $(V, W)$ such that $W\subset V$ with $V\in \mbox{Rep}(\vec{Q}, \alpha+\beta)$ and $W\in \mbox{Rep}(\vec{Q}, \alpha)$. Let $Corr_2$ be the variety consisting of all quadruples $(V, W, i, j)$ where $(V, W)\in Corr_1$ and the maps $i:\mathbb{C}^{\alpha}\rightarrow W$ and $j:\mathbb{C}^{\beta}\rightarrow \mathbb{C}^{\alpha}/W$ are isomorphisms. Define projections
\begin{align} p_2: Corr_2&\rightarrow Corr_1 \nonumber \\
(V, W, i, j)&\mapsto (V, W) \nonumber \\
\nonumber \\
p_3: Corr_1&\rightarrow \mbox{Rep}(\vec{Q}, \alpha+\beta) \nonumber \\
(V, W)&\mapsto V \nonumber \\
\nonumber\\
p_1:Corr_2&\rightarrow \mbox{Rep}(\vec{Q}, \alpha)\times \mbox{Rep}(\vec{Q}, \beta)\nonumber \\
(V,W, i, j) &\mapsto (V', V'') \nonumber
\end{align}
where $V(a) i_{t(a)}= i_{s(a)} V'(a)$ and $V(a) j_{t(a)} = j_{s(a)} V''(a)$ for all arrows $a\in Q_1$.
The convolution product of $f_{\alpha}$ and $f_{\beta}$ is then defined as
\[f_{\alpha}* f_{\beta}=(p_3)_{*}(p_2)_{\flat}(p_1)^*(f_{\alpha}\times f_{\beta})\]
Given three representations $X, Y$, and $Z$, define a complex variety
\[G_{X, Y}^Z=\{W\subset Z| W\cong X, Z/W\cong Y\}\]
Let $[X], [Y]$, and $[Z]$ denote the characteristic functions of $X, Y$, and $Z$. Then the multiplicity of $[Z]$ in the product $[X]*[Y]$ is the Euler characteristic $\chi(G_{X,Y}^Z)$.

The commutator $[f_{\alpha}, f_{\beta}]=f_{\alpha}*f_{\beta} - f_{\beta}*f_{\alpha}$ defines a Lie bracket on $\mathcal{L}(\vec{Q}).$ Let $M_{\vec{Q}}$ denote the Lie subalgebra of $\mathcal{L}(\vec{Q})$ generated by the characteristic functions $\{[S_i] | i\in Q_0\}.$ It can be shown that $f([M])=0$ for any constructible function $f\in M_{\vec{Q}}$ and any decomposable representation $M\in \mathbb{C}\vec{Q}$-mod (\cite{FMV1, Rie}). One can check that the characteristic functions $\{[S_i]\}$ satisfy the Serre relations
\[(ad([S_i]))^{1-c_{i, j}}[S_j]=0\]
for $i\ne j$.

For $\alpha\in \Delta^{re}(\vec{Q})$, let $E_{\alpha}$ denote the characteristic function of the indecomposable representations of $\vec{Q}$ in dimension $\alpha$. Fix an extending vertex $e\in Q_0$ and let $\alpha_0=\underline{dim}(P(e))$. The choice of extending vertex gives us a parameterization of the tubes in $Rep(\vec{Q})$. Let $z_1, \cdots, z_L\in \mathbb{CP}^1$ denote the points corresponding to the non-homogeneous tubes. Let $\alpha_{i, j}=\underline{dim}(P_{z_i, j, 1})$ and let $E_{i, j}(m)$ denote the characteristic function of the indecomposables $P_{z_i, j, mN_{z_i}}$. Finally, let $E_0(m)$ denote the characteristic function of the regular indecomposables in dimension $m\delta$ with top supported on $e$.

\begin{theorem} (\cite{FMV1}) The assignment
\begin{align} \tilde{e}_{\alpha}&\mapsto E_{\alpha} \nonumber \\ 
\alpha_{i, j}(m)&\mapsto E_{i, j}(m)-E_{i, j+1}(m) \nonumber \\
\alpha_0(m)&\mapsto -E_0(m) \nonumber
\end{align}
defines an isomorphism $\mathfrak{n}^{\epsilon^*}(\vec{Q})\xrightarrow{\sim} M_{\vec{Q}}$

\end{theorem}

Now consider the category of representations of $\Pi_Q$. Define the convolution algebra $\mathcal{L}(\Pi_Q)$ analogously to the construction of $\mathcal{L}(\vec{Q})$, i.e. by replacing the representation spaces Rep$(\vec{Q}, \alpha)$ with the spaces Rep$^{nil}(\Pi_Q, \alpha)$, etc. Denote the subalgebra of $\mathcal{L}(\Pi_Q)$ generated by the characteristic functions $\{[S_i]\}$ by $H_Q.$ In this case, one can again check that the generators $\{[S_i]\}$ satisfy the Serre relations
\[(ad([S_i]))^{1-c_{i, j}}[S_j]=0\]
for $i\ne j$. Any function $f\in \mathcal{L}_{\alpha}(\Pi_Q)$ in $H_Q$ is constant on a dense open subset of each irreducible component in $\Lambda_{\alpha}$. Moreover, for each irreducible component $Z\in Irr \Lambda_{\alpha}$, there exists a function $f_Z\in H_Q\cap \mathcal{L}_{\alpha}(\Pi_Q)$ that takes the value $1$ on a dense open subset of $Z$ and $0$ on a dense open subset of each irreducible component $Z'\in Irr \Lambda_{\alpha}- \{Z\}$ (\cite{L3}). For each $\alpha\in \mathbb{Z}^{Q_0}$, the functions $\{f_Z\}_{Z\in Irr \Lambda_{\alpha}}$ form a basis for $H_Q\cap \mathcal{L}_{\alpha}(\Pi_Q)$.

Note that any representation $X\in \Pi_Q$-mod naturally determines a representation of $\vec{Q}$. This is obtained by letting all the ``reverse" arrows act as 0. Denote this map by $\rho:\Pi_Q \mbox{-mod}\rightarrow \mathbb{C}\vec{Q}\mbox{-mod}$. For a constructible function $f\in \mathcal{L}(\Pi_Q)$, we can consider its restriction to representations $X\in \Pi_Q$-mod satisfying $\rho(X)=X$. Identifying these representations with representations of $\vec{Q}$, we denote this new function by $\rho_*(f)\in \mathcal{L}(\vec{Q})$. Note that if $X\in \Pi_Q$-mod satisfies $\rho(X)=X$, then any subobject or quotient $Y$ of $X$ in $\Pi_Q$-mod must also satisfy $\rho(Y)=Y$. It follows that the map $\rho_*:\mathcal{L}(\Pi_Q)\rightarrow \mathcal{L}(\vec{Q})$ is an epimorphism. Moreover, one can show that $\rho_*$ restricts to an isomorphism $\rho_*:H_Q\xrightarrow{\sim} M_{\vec{Q}}$.

\subsection{Semistable construction of $\mathfrak{n}$}

In this section, we give a construction of $\mathfrak{n}$ using the semistable representations of $\Pi_Q$. We first study how the representations of $\vec{Q}$ lift to representations of $\Pi_Q$.

\begin{lemma} A short exact sequence
\[0\rightarrow X\rightarrow Z\rightarrow Y\rightarrow 0\]
in $\Pi_Q\mbox{-mod}$ with $\rho(X)=X$ and $\rho(Y)=Y$ naturally determines a short exact sequence
\[0\rightarrow X\rightarrow \rho(Z)\rightarrow Y\rightarrow 0\]
in $\mathbb{C}\vec{Q}\mbox{-mod}$.

\end{lemma}

\begin{proof} This follows easily from the definitions of $\mathbb{C}\vec{Q}$-mod and $\Pi_Q$-mod.

\end{proof}

\begin{corollary}
If $\rho(X)=X, \rho(Y)=Y,$ and $[Z]$ appears in the Hall algebra product of $[X]$ and $[Y]$ in $H_{\vec{Q}}$, then $[\rho(Z)]$ appears in the product of $[X]$ and $[Y]$ considered inside $M_{\vec{Q}}$. Moreover, the multiplicity of $[\rho(Z)]$ in the second product is greater than or equal to the multiplicity of $[Z]$ in the first product.

\end{corollary}

\begin{corollary} If $X\in \Pi_Q\mbox{-mod}$ has a composition series $0\subset F_1\subset F_2 \cdots \subset F_n \subset X$ such that $F_1=\rho(F_1), F_2/F_1=\rho(F_2/F_1), \cdots, F_n/F_{n-1}=\rho(F_n/F_{n-1})$, then the projection $\rho(X)$ has the same composition series. 

\end{corollary}

\begin{proof} For $n=2$, the statement follows from Lemma 1. The general result follows by induction.

\end{proof}

Let $\mathcal{A}$ be an abelian category, let $F:\mathcal{A}\rightarrow \mathcal{A}$ be an endofunctor, and let $X\in Ob(\mathcal{A})$ be an object. The orbit algebra of $\mathcal{O}^F(X)$ is defined by
\[\mathcal{O}^F(X)=\oplus_{n\ge 0} \mbox{Hom}(F^n(X), X)\]
with multiplication given by
\[f*g=f\circ F^s(g)\]
for $f\in \mbox{Hom}(F^s(X), X)$ and $g\in \mbox{Hom}(F^t(X), X)$. Define
\[\mathcal{N}^F(X)=\{f\in \mbox{Hom}(F(X), X)| f\mbox{ is nilpotent in }\mathcal{O}^F(X)\}\]
The following can be found in \cite{Ri}:
\begin{proposition} (\cite{Ri}) If $X\in \mathbb{C}\vec{Q}$-mod, then $\rho^{-1}(X)\cong \mathcal{N}^{\tau}(X)$ for $\tau$ the Auslander-Reiten translation of $\mathbb{C}\vec{Q}$-mod.

\end{proposition}

Now we consider different stability conditions on the categories $\mathbb{C}\vec{Q}$-mod and $\Pi_Q$-mod. Note that any homomorphism $\lambda:K_0(\mathbb{Z}^{Q_0})\rightarrow \mathbb{Z}$ defines a stability condition on both categories. Moreover, if $X\in \mathbb{C}\vec{Q}$-mod is $\lambda$-semistable (respectively, $\lambda$-stable) then it is $\lambda$-semistable (respectively, $\lambda$-stable) when considered as a $\Pi_Q$-module. A root $\alpha\in \Delta(\vec{Q})$ is said to be a \textit{Schur root} if there exists a brick $X\in \mathbb{C}\vec{Q}$-mod such that $\underline{dim}(X)=\alpha$. Because $\mathbb{C}\vec{Q}$-mod is hereditary, the moduli space Rep$^s(\mathbb{C}\vec{Q}, \alpha)/_{\lambda} GL(\alpha)$ is non-empty for some stability condition $\lambda$ if and only if $\alpha$ is a Schur root (\cite{Ki}). In the case that $\alpha$ is a Schur root, any indecomposable representation in dimension $\alpha$ is stable with respect to the \textit{canonical character} (\cite{Scho})
\begin{align} \lambda_{\alpha}:K_0(\mathbb{Z}^{Q_0})&\rightarrow \mathbb{Z} \nonumber \\ 
\underline{dim}(V) &\mapsto e(\alpha, \underline{dim}(V))- e(\underline{dim}(V), \alpha)\nonumber
\end{align}
It is easy to see that each preprojective and preinjective indecomposable representation of $\vec{Q}$ is a brick. It follows from \cite[Theorem 1.1]{CB} that for any $\alpha\in \Delta^{irr}(\vec{Q})$, all $\lambda_{\alpha}$-stable $\Pi_Q$-modules in dimension $\alpha$ are isomorphic to $P_{\alpha}$.

For regular roots, we define a stability condition
\begin{align} \lambda_{reg}:K_0(\Pi_Q\mbox{-mod})&\rightarrow \mathbb{Z} \nonumber \\ 
\underline{dim}(V)&\mapsto e(\delta, \underline{dim}(V))\nonumber
\end{align}
By Proposition 1, each regular $\mathbb{C}\vec{Q}$-module is $\lambda_{reg}$-semistable. Moreover, by Proposition 3, each regular simple object in $\mathbb{C}\vec{Q}$-mod lifts to a unique $\Pi_Q$-module. These are the $\lambda_{reg}$-stable objects in $\mathbb{C}\vec{Q}$-mod and $\Pi_Q$-mod. It follows from Corollary 2 that the $\lambda_{reg}$-semistable objects in $\Pi_Q$-mod are exactly the modules $X$ such that $\rho(X)$ is a regular $\mathbb{C}\vec{Q}$-module. Moreoever, any such $X$ is S-equivalent to $\rho(X)$.

In \cite{FMV2}, the authors choose a generic stability condition for the dimension $\delta$. They verify that a statement analogous to Theorem 1 holds in type $A_1^{(1)}$. In this case, there are no non-homogeneous tubes, and their choice of stability conditions coincides with ours. In general, $\lambda_{reg}$ is not generic, and in fact it lies on the maximal number of walls. For types $A_2^{(1)}$, $A_3^{(1)}$, and $D_4^{(1)}$, Theorem 1 still holds when $\lambda_{reg}$ is replaced by a generic stability condition. However, for quivers with tubes of rank $\ge 3$, our choice of stability condition is necessary to ensure an isomorphism $\mathfrak{n}\xrightarrow{\sim} H^{ss}_{\vec{Q}}$. We now give a proof of Theorem 1:

\begin{proof}
Let $H^{ss}_{\vec{Q}}$ denote the subspace of $H_{\vec{Q}}$ spanned by constructible functions supported on semistable irreducible components. Semistability is defined here using the stability conditions $\lambda_{reg}$ for regular roots and $\lambda_{\alpha}$ for each irregular root $\alpha$. A Lie bracket is defined on $H^{ss}_{\vec{Q}}$ by restriction of the bracket in $H_{\vec{Q}}$. We follow a similar strategy to \cite{FMV1} to construct an isomorphism of Lie algebras $\Phi:\mathfrak{n}^{\epsilon^*}(\vec{Q})\xrightarrow{\sim} H^{ss}_{\vec{Q}}$. This map lifts the isomorphism $\Xi:\mathfrak{n}^{\epsilon^*}(\vec{Q})\xrightarrow{\sim} M_{\vec{Q}}$ in the sense that $\rho_*\circ \Phi = \Xi$.

For $\alpha$ an irregular root, let $\widehat{E}_{\alpha}$ be the characteristic function of the isomorphism class of indecomposable representations in dimension $\alpha$ and define $\Phi(e_{\alpha})=\widehat{E}_{\alpha}$. If $\alpha$ and $\beta$ are two irregular roots such that $\alpha+\beta\in \Delta^{irr}(\vec{Q})$, then it is clear that
\[\Phi([e_{\alpha}, e_{\beta}])=[\Phi(e_{\alpha}), \Phi(e_{\beta})]\]

Fix an extending vertex $e$ and let the points $z_1, \cdots, z_L\in \mathbb{P}^1(\mathbb{C})$ index the non-homogeneous tubes in $\mathbb{C}\vec{Q}$-mod. Let $\widehat{E}_{i, j}$ denote the characteristic function of the isomorphism class of regular simple objects $P_{z_i, j, 1}$ considered as representations of $\Pi_Q$ and let $\alpha_{i, j}=\underline{dim}(P_{z_i, j, 1})$. In \cite{FMV1} it is shown that the corresponding elements $E_{i, j}=\rho_*(\widehat{E}_{i, j})$ in $M_{\vec{Q}}$ can be obtained by taking Lie brackets of elements $E_{\alpha}$ for irregular roots $\alpha$. Extend $\Phi$ by setting $\Phi(e_{\alpha_{i, j}})=\widehat{E}_{i, j}$. 

In non-homogeneous tubes indexed by $z_1, \cdots, z_L$, the characteristic functions of each isomorphism class $P_{z_i, j, l}$ can be uniquely expressed as Lie polynomials in $M_{\vec{Q}}$ in the elements $E_{i, j}$. Denote the corresponding Lie polynomials in $H^{ss}_{\vec{Q}}$ by $\widehat{E}_{i, j, l}$ and set $\widehat{E}_{i, j}(m)=\widehat{E}_{i, j, mN_{z_i}}$. We extend $\Phi$ by setting $\Phi(e_{\alpha_{i, j, l}})=\widehat{E}_{i, j, l}$ and $\Phi(\alpha_{i, j}(m))=\widehat{E}_{i, j}(m)-\widehat{E}_{i, j+1}(m)$, where $\alpha_{i, j, l}=\underline{dim}(P_{z_i, j, l})$. The elements $\widehat{E}_{i, j, l}$ are well-defined because there are no relations between characteristic functions corresponding to different tubes.

Finally, let $\alpha_0=\underline{dim}(P(e))$. Define an element
\[\widehat{E}_0(m)=\epsilon^{*}(\alpha_0, m\delta-\alpha_0)[\widehat{E}_{\alpha_0}, \widehat{E}_{m\delta-\alpha_0}]\] 
Note that this bracket is defined by restricting the bracket in $H_{\vec{Q}}$ to semistable irreducible components. This is equivalent to restricting to modules $X\in \Pi_Q$-mod such that $\rho(X)$ is indecomposable. Define $\Phi(\alpha_0(m))= \widehat{E}_0(m)$.

The elements $\{\alpha_{i, j}(m)\}$ and $\alpha_0(m)$ form a basis for $\mathfrak{n}^{\epsilon^*}_{m\delta}(\vec{Q})$. One can now easily check that the assignments
\begin{align} \tilde{e}_{\alpha}& \mapsto \widehat{E}_{\alpha} \nonumber \\ 
\alpha_{i, j}(m)& \mapsto \widehat{E}_{i, j}(m)-\widehat{E}_{i, j+1}(m)\nonumber \\
\alpha_0(m)& \mapsto -\widehat{E}_0(m) \nonumber
\end{align}
extend to an isomorphism $\Phi:\mathfrak{n}^{\epsilon^*}(\vec{Q})\xrightarrow{\sim} H^{ss}_{\vec{Q}}$. By Theorem 4, this gives us an isomorphism $\mathfrak{n}\xrightarrow{\sim} H^{ss}_{\vec{Q}}$.

\end{proof}

It follows from Corollaries 1 and 2 that the functions $\widehat{E}_0(m)$ are supported on $\Pi_Q$-modules that are S-equivalent to the regular modules in the support of $E_0(m)$. Similarly, the functions $\widehat{E}_{i, j, l}$ are supported on semistable $\Pi_Q$-modules that are S-equivalent to $P_{z_i, j, l}$. In dimension $m\delta$, these S-equivalence classes are exactly the \textit{semistable diagonals} of \cite{FMV2} and \cite{HM}.

\bibliographystyle{amsalpha}

\begin{thebibliography}{A}

\bibitem[1]{Be}

K. Behrend: The Lefschetz trace formula for algebraic stacks. Invent. Math. 112 (1993), no. 1, 127-149.

\bibitem[2]{BGP}

J. Bernstein, I. Gelfand, and V. Ponomarev: Coxeter functors, and Gabriel's theorem. Uspehi Mat. Nauk 28 (1973), no. 2(170), 19-33.

\bibitem[3]{BCY}

J. Bryan, C. Cadman, B. Young: The orbifold topological vertex, Advances in Mathematics 229.1 (2012), 531-595

\bibitem[4]{CS}

H. Cassens and P. Slodowy: On Kleinian singularities and quivers. Singularities (Oberwolfach 1996), 263-288. Progr. Math 162.

\bibitem[5]{C}

T. Cramer: Double Hall algebras and derived equivalences. Advances in Mathematics 224.3 (2010): 1097-1120.

\bibitem[6]{CB}

W. Crawley-Boevey: Geometry of the moment map for representations of quivers. Compositio Math. 126 (2001), no. 3, 257-293.

\bibitem[7]{CBVdB}

W. Crawley-Boevey and M. Van Den Bergh: Absolutely indecomposable representations and Kac-Moody Lie algebras. With an appendix by Hiraku Nakajima. Invent. Math. 155 (2004), no. 3, 537-559.

\bibitem[8]{DF}

P. Donovan and M. R. Freislich: The representation theory of finite graphs and associated algebras. Carleton University, Ottawa, Ont., 1973, Carleton Mathematical Lecture Notes, No. 5.

\bibitem[9]{DR}

V. Dlab and C. M. Ringel: Indecomposable representations of graphs and algebras. Mem. Amer. Math. Soc. 6 (1976), no. 173, v+57.

\bibitem[10]{FM}

B. Fiol and M. Mari\~{n}o: BPS states and algebras from quivers. J. High Energy Phys. 2000, no. 7, Paper 31, 40 pp.

\bibitem[11]{FK}

I. B. Frenkel and V. G. Kac: Basic representations of affine Lie algebras and dual resonance models. Invent. Math. 62 (1980), no. 1, 23-66.

\bibitem[12]{FMV1}

I. Frenkel, A. Malkin, and M. Vybornov, Affine Lie Algebras and Tame Quivers, Selecta Math. (N.S.) 7 (2001), no. 1, 1-56.

\bibitem[13]{FMV2}

I. Frenkel, A. Malkin, M. Vybornov: Quiver varieties, affine Lie algebras, algebras of BPS states, and
semicanonical basis. Algebraic combinatorics and quantum groups. World Scientific, 2003.

\bibitem[14]{Gab}

P. Gabriel: Unzerlegbare Darstellungen. I, Manuscripta Math. 6 (1972), 71-103; correction, ibid. 6 (1972), 309.

\bibitem[15]{HM}

J. Harvey and G. Moore: On the algebras of BPS states. Comm. Math. Phys. 197 (1998), no. 3, 489-519.

\bibitem[16]{Ha}

T. Hausel: Kac's conjecture from Nakajima quiver varieties. Invent. Math. 181 (2010), no. 1, 21-37. MR 2651380

\bibitem[17]{HLRV}

T. Hausel, E. Letellier, and F. Rodriguez-Villegas: Positivity of Kac polynomials and DT-invariants for quivers. arXiv preprint arXiv:1204.2375 (2012).

\bibitem[18]{K1}

V. Kac: Infinite root systems, representations of graphs and invariant theory. Invent. Math. 56 (1980), 57-92

\bibitem[19]{K2}

V. Kac: Root systems, representations of quivers and invariant theory. Invariant theory (Montecatini, 1982), 74-108, Lecture Notes in Mathematics, 996, Springer Verlag 1983

\bibitem[20]{KV}

M. Kapranov and E. Vasserot: Kleinian singularities, derived categories and Hall algebras. Mathematische Annalen 316.3 (2000): 565-576.

\bibitem[21]{KaS}

M. Kashiwara and Y. Saito: Geometric construction of crystal bases. Duke Mathematical Journal 89.1 (1997): 9-36.

\bibitem[22]{Kim}

M. Kim: A Lefschetz trace formula for equivariant cohomology. Ann. Sci. \'{E}cole Norm. Sup.
(4) 28 (1995), no. 6, 669-688.

\bibitem[23]{Ki}

A. King: Moduli of representations of finite dimensional algebras. Quart. J. Math. Oxford (2), 45, (1994), 515-530.

\bibitem[24]{KS}

M. Kontsevich and Y. Soibelman: Cohomological Hall algebra, exponential Hodge structures and motivic Donaldson-Thomas invariants. arXiv preprint arXiv:1006.2706 (2010).

\bibitem[25]{Kr}

P. Kronheimer: The construction of ALE spaces as hyper-Kahler quotients. J. Differential Geom. 29 (1989), 665-683.

\bibitem[26]{L}

G. Lusztig: Quivers, perverse sheaves, and quantized enveloping algebras. J. Amer. Math. Soc. 4 (1991), no. 2, 365-421.

\bibitem[27]{L2}

G. Lusztig: Affine quivers and canonical bases. Inst. Hautes Etudes Sci. Publ. Math. (1992), no. 76, 111-163.

\bibitem[28]{L3}

G. Lusztig: Semicanonical bases arising from enveloping algebras. Advances in Mathematics 151.2 (2000): 129-139.

\bibitem[29]{Na}

L. A. Nazarova: Representations of quivers of infinite type. Izv. Akad. Nauk SSSR Ser. Mat. 37 (1973), 752-791.

\bibitem[30]{Re}

M. Reineke: The Harder-Narasimhan system in quantum groups and cohomology of quiver moduli. Inventiones mathematicae 152.2 (2003): 349-368.

\bibitem[31]{Rie}

C. Riedtmann: Lie algebras generated by indecomposables. J. Algebra 170 (1994), no. 2, 526-546.

\bibitem[32]{Ri}

C. M. Ringel: The preprojective algebra of a quiver. Algebras and modules II., Eighth international conference on representations of algebras (Geiranger, Norway, 1996), ed. by I. Reiten, CMS conf. Proc. vol. 24, Amer. Math. Soc., 1998, pp. 467-480.

\bibitem[33]{Ru}

A. Rudakov: Stability for an abelian category. J. Algebra 197 (1997), no. 1, 231-245.

\bibitem[34]{Scho}

A. Schofield: General representations of quivers. Proc. London Math. Soc. (3) 65 (1992), 46-64.

\end{thebibliography}

\end{document}